\documentclass[onefignum,onetabnum]{siamart171218}

\usepackage{lipsum}
\usepackage{amsfonts}
\usepackage{graphicx}
\usepackage{epstopdf}
\usepackage{algorithmic}
\ifpdf
  \DeclareGraphicsExtensions{.eps,.pdf,.png,.jpg}
\else
  \DeclareGraphicsExtensions{.eps}
\fi

\newsiamremark{remark}{Remark}
\newsiamremark{hypothesis}{Hypothesis}
\crefname{hypothesis}{Hypothesis}{Hypotheses}
\newsiamthm{claim}{Claim}

\headers{Log-Cholesky Metric}{Z. Lin}

\title{Riemannian Geometry of Symmetric Positive Definite Matrices via Cholesky
Decomposition\thanks{Submitted to the editors on  \today.}}

\author{Zhenhua Lin\thanks{National University of Singapore 
  (\email{stalz@nus.edu.sg}, \url{https://blog.nus.edu.sg/zhenhua/}).}
}

\usepackage{amsopn}

\ifpdf
\hypersetup{
  pdftitle={Riemannian Geometry of Symmetric Positive Definite Matrices via Cholesky
Decomposition},
  pdfauthor={Z. Lin}
}
\fi

\usepackage{mathrsfs}
\usepackage{amsfonts}
\usepackage{stmaryrd}
\usepackage{xargs}[2008/03/08]

\theoremstyle{plain}
\newtheorem{thm}{Theorem}
  \theoremstyle{definition}
  \newtheorem{example}[thm]{Example}
  \theoremstyle{plain}
  \newtheorem{prop}[thm]{Proposition}
  \theoremstyle{plain}
  \newtheorem{lem}[thm]{Lemma}
  \theoremstyle{plain}
  \newtheorem{cor}[thm]{Corollary}

\global\long\def\expect{\mathbf{\mathbb{E}}}

\global\long\def\real{\mathbb{R}}

\global\long\def\manifold{\mathcal{M}}

\newcommandx\manball[1][usedefault, addprefix=\global, 1=\delta]{\mathbb{B}_{#1}^{\manifold}(x)}
\global\long\def\diffop{\mathrm{d}}

\global\long\def\cholspace{\mathcal{L}_{+}}
\global\long\def\lowtri{\mathcal{L}}

\global\long\def\transpose{\top}

\global\long\def\ch#1{\mathscr{L}(#1)}
\global\long\def\lp#1{\lfloor#1\rfloor}
\global\long\def\dg#1{\mathbb{D}(#1)}
\global\long\def\dgmap{\mathbb{D}}
\global\long\def\differential#1#2{D_{#2}#1}
\global\long\def\trace{\mathrm{tr}}
\global\long\def\mygroupop{\varocircle}

\global\long\def\spdgpop{\varoast}
\global\long\def\chmap{\mathscr{L}}
\global\long\def\chtospd{\mathscr{S}}

\global\long\def\Exp{\mathrm{Exp}}
\newcommandx\spd[1][usedefault, addprefix=\global, 1=m]{\mathcal{S}_{#1}^{+}}
\newcommandx\tangentspace[2][usedefault, addprefix=\global, 1=\manifold]{T_{#2}#1}
\newcommandx\sym[1][usedefault, addprefix=\global, 1=m]{\mathcal{S}_{#1}}
\global\long\def\define{:=}
\newcommandx\fronorm[2][usedefault, addprefix=\global, 1=]{\|#2\|_{\mathrm{F}}^{#1}}
\global\long\def\innerprod#1#2{\langle#1,#2\rangle}

\global\long\def\Log{\mathrm{Log}}
\newcommandx\froinnerprod[3][usedefault, addprefix=\global, 1=]{\innerprod{#2}{#3}_{\mathrm{F}}^{#1}}
\global\long\def\gspd{g}
\global\long\def\gchol{\tilde{g}}
\global\long\def\half{\frac{1}{2}}
\global\long\def\geospd{\gamma}
\global\long\def\geochol{\tilde{\gamma}}
\global\long\def\respd{\Exp}
\global\long\def\rechol{\widetilde{\Exp}}
\global\long\def\logchol{\widetilde{\mathrm{Log}}}

\newcommand{\linremark}[1]{{#1}}

\begin{document}

\maketitle

\begin{abstract}
We present a new Riemannian metric, termed Log-Cholesky metric, 
on the manifold of symmetric positive definite
(SPD) matrices via Cholesky decomposition. We first construct a Lie
group structure and a bi-invariant metric on Cholesky space, the collection
of lower triangular matrices whose diagonal elements are all positive.
Such group structure and metric are then pushed forward to the space
of SPD matrices via the inverse of Cholesky decomposition that is
a bijective map between Cholesky space and SPD matrix space. This new Riemannian metric and Lie group structure fully circumvent swelling
effect, in the sense that the determinant of the Fr\'echet average of a set of SPD matrices under
the presented metric, called Log-Cholesky average, is between the minimum and the maximum of the determinants
of the original SPD matrices. Comparing to existing metrics such as the affine-invariant
metric and Log-Euclidean metric, the presented metric is simpler,
more computationally efficient and numerically stabler. In particular,  parallel transport along geodesics under Log-Cholesky metric is given in a closed and easy-to-compute form.
\end{abstract}



\section{Introduction}

Symmetric positive definite (SPD) matrices emerge in vast scientific
applications such as computer vision \cite{Caseiro2012,Rathi2007},
elasticity \cite{Guilleminot2012,Moakher2006}, signal processing
\cite{Arnaudon2013,Hua2017}, medical imaging \cite{Dryden2009,Fillard2007,Fletcher2007,Lenglet2006,Wang2004}
and neuroscience \cite{Friston2011}. A concrete example is analysis
of functional connectivity between brain regions. Such connectivity
is often characterized by the covariance of blood-oxygen-level dependent
signals \cite{Huettel2008} generated by brain activities from different
regions. The covariance is mathematically defined by a covariance
matrix which is an SPD matrix. Another application is diffusion tensor
imaging \cite{LeBihan1991}, which is extensively used to obtain
high-resolution information of internal structures of certain tissues
or organs, such as hearts and brains. For each tissue voxel, there is
a $3\times3$ SPD matrix to describe the shape of local diffusion.
Such information has clinical applications; for example, it can be used
to discover pathological area surrounded by healthy tissues.

The space of SPD matrices of a fixed dimension $m$, denoted by $\spd$
in this article, is a convex smooth submanifold of the Euclidean space
$\real^{m(m+1)/2}$. The inherited Euclidean metric further turns
$\spd$ into a Riemannian manifold. However, as pointed out in \cite{Arsigny2007},
this classic metric is not adequate in many applications for two reasons.
First, the distance between SPD matrices and symmetric matrices with
zero or negative eigenvalues is finite, which implies that, in the
context of diffusion tensor imaging, small diffusion is more likely
than large diffusion. Second, the Euclidean average of SPD matrices
suffers from swelling effect, i.e., the determinant of the average
is larger than any of the original determinants. When SPD matrices
are covariance matrices, as in the application of diffusion tensor
imaging,  determinants correspond to overall dispersion of 
diffusion. Inflated determinants amount to extra diffusion that is 
artificially introduced in computation. 

To circumvent the problems of the Euclidean metric for SPD matrices,
various metrics have been introduced in the literature, such as the
affine-invariant metric \cite{Moakher2005,Pennec2006a} and the Log-Euclidean
metric \cite{Arsigny2007}. These metrics keep symmetric matrices
with some nonpositive eigenvalues at an infinite distance away from
SPD matrices, and are not subject to swelling effect. In addition,
the Log-Euclidean framework features a closed form of the Fr\'echet average
of SPD matrices. It also turns $\spd$ into a Lie group endowed with
a bi-invariant metric. However, computation of Riemannian exponential
and logarithmic maps requires evaluating a series of an infinite number
of terms; see Eq. (2.1) and (3.4) in \cite{Arsigny2007}. Comparing
to the Log-Euclidean metric, the affine-invariant one not only possesses
easy-to-compute exponential and logarithmic maps, but also enjoys
a closed form for parallel transport along geodesics; see Lemma 3
of \cite{Schiratti2017}. However, to the best of our knowledge, no closed  form is found for the Fr\'echet
average of SPD matrices under the affine-invariant metric. \linremark{The Fr\'echet average of SPD matrices is also studied in the literature for distance functions or Riemannian metrics arising from perspectives other than swelling effect, such as the Bures-Wasserstein   metric that is related to the theory of optimal transport \cite{Bhatia2018}, and the S-divergence studied in both \cite{Chebbi2012} and \cite{Sra2016}. Other related works include Riemannian geometry for positive semidefinite matrices \cite{Vandereycken2013,Massart2018} and  Riemannian structure for correlation matrices \cite{Grubisic2007}.}

In addition to the above Riemannian frameworks, it is also common
to approach SPD matrices via Cholesky decomposition in practice for efficient computation, such
as \cite{Eubank2002,Osborne2013,Wang2004}. Distance on SPD matrices
based on Cholesky decomposition has also been explored in the literature.
For example, in \cite{Dryden2009} the distance between two SPD
matrices $P_{1}$ and $P_{2}$ with Cholesky decomposition $P_1=L_1L_1^\transpose$ and $P_2=L_2L_2^\transpose$ is defined by $\fronorm{L_1-L_2}$, where each of $L_1$ and $L_2$ is
a lower triangular matrix whose diagonal elements are positive, 
and $\fronorm{\cdot}$ denotes Frobenius matrix norm. Although this distance is simple and easy to
compute, it suffers from swelling effect, as demonstrated by the
following example.
\begin{example}One first notes that, under the Cholesky distance, the geodesic interpolation between $P_{1}$
and $P_{2}$ is given by $P_{\rho}\define\{\rho L_{1}+(1-\rho)L_{2}\}\{\rho L_{1}+(1-\rho)L_{2}\}^{\transpose}$
for $\rho\in[0,1]$. 
For any $\epsilon>0$, consider matrices 
\[
P_{1}=\begin{pmatrix}\epsilon^{2} & 0\\
0 & 1
\end{pmatrix},\qquad P_{2}=\begin{pmatrix}1 & 0\\
0 & \epsilon^{2}
\end{pmatrix},\qquad L_{1}=\begin{pmatrix}\epsilon & 0\\
0 & 1
\end{pmatrix},\qquad L_{2}=\begin{pmatrix}1 & 0\\
0 & \epsilon
\end{pmatrix}.
\]
It is clear that $L_{1}L_{1}^{\transpose}=P_{1}$ and $L_{2}L_{2}^{\transpose}=P_{2}$.
When $\rho=1/2$, 
\[
P_{\rho}=\frac{1}{4}(L_{1}+L_{2})(L_{1}+L_{2})^{\transpose}=\begin{pmatrix}\frac{(1+\epsilon)^{2}}{4} & 0\\
0 & \frac{(1+\epsilon)^{2}}{4}
\end{pmatrix},
\]
whose determinant is $\det(P_{\rho})=(1+\epsilon)^{4}/16$. However,
$\det(P_{1})=\det(P_{2})=\epsilon^{2}<(1+\epsilon)^{4}/16$, or equivalently,
$\max\{\det(P_{1}),\det(P_{2})\}<\det(P_{\rho})$, whenever $\epsilon\neq1$. 
\end{example}

In this work, we propose a new Riemannian metric on SPD matrices via
Cholesky decomposition. \linremark{The basic idea is to introduce a new metric for the space of lower triangular matrices with positive diagonal elements and then push it forward to the space of SPD matrices via Cholesky decomposition.} The metric, termed \emph{Log-Cholesky metric}, has
the advantages of the aforementioned affine-invariant metric, Log-Euclidean
metric and Cholesky distance. First, it is as simple as the Cholesky
distance, but not subject to swelling effect. Second, like the Log-Euclidean
metric, the presented metric enjoys Lie group bi-invariance, as well
as a closed form for the Log-Cholesky average of SPD matrices. \linremark{This bi-invariant Lie group structure seems not shared by the aforementioned works other than \cite{Arsigny2007} in the literature.} Third, it
features simple and easy-to-compute expressions for Riemannian exponential
and logarithmic maps, in contrast with the Log-Euclidean metric. Finally,
like the affine-invariant metric, the expression for parallel transport
along geodesics is simple and easy-to-compute under the presented metric. Parallel
transport is important in applications like regression methods on
Riemannian manifolds, such as \cite{Schiratti2017,Yuan2012,Zeestraten2017}. 

\linremark{It is noted that Cholesky decomposition is also explored in \cite{Grubisic2007} for a Riemannian geometry of correlation matrices with  rank no larger than a fixed bound. Despite certain similarity in the use of Cholesky decomposition, this work is fundamentally different from ours. First, it studies correlation matrices rather than SPD matrices. For a correlation matrix, its diagonal elements are restricted to be one. Second, the Riemannian structures considered in \cite{Grubisic2007} and our work are different. For example, the so-called Cholesky manifold in \cite{Grubisic2007} is a Riemannian submanifold of a Euclidean space, while our Riemannian manifold to be proposed is not. Finally, Cholesky decomposition is utilized in \cite{Grubisic2007} as a way to parameterize correlation matrices, rather than push forward a new manifold structure to correlation matrices.}

We structure the rest of this article as follows. Some notations and basic properties of lower triangular and SPD matrices are collected in \cref{sec:Cholesky-decomposition}. In \cref{sec:Lie-group-structure},
we introduce a new Lie group structure on SPD matrices and define the Log-Cholesky
metric on the group. Basic features such as Riemannian exponential/logarithmic
maps, geodesics and parallel transport are also characterized. \Cref{sec:Mean-of-distributions} is devoted to the Log-Cholesky mean/average 
of distributions on SPD matrices. 
We then conclude the article in \cref{sec:Conclusion}.

\section{Lower triangular matrices and SPD matrices\label{sec:Cholesky-decomposition}}
We start with introducing some notations and recalling some basic properties of lower triangular and SPD matrices. Cholesky decomposition is then shown to be a diffeomorphism between lower triangular matrix manifolds and SPD manifolds. This result serves as a cornerstone of our development: it enables us to push forward a Riemannian metric defined on the space of triangular matrices to the space of SPD matrices.

\subsection{Notations and basic properties}\label{subsec:lower}
Throughout this paper, $m$ is a fixed positive integer that represents the dimension of matrices under consideration. 
For a matrix $A$, we use $A_{ij}$ or $A(i,j)$ to denote its element on the
$i$th row and $j$th column. The notation $\lp A$ denotes an $m\times m$
matrix whose $(i,j)$ element is $A_{ij}$ if $i>j$ and is zero otherwise,
while $\dg A$ denotes an $m\times m$ diagonal matrix whose $(i,i)$ element is
$A_{ii}$. In other words, $\lp A$
is the strictly lower triangular part, while $\dg A$
is the diagonal part of $A$. The
trace of a matrix $A$ is denoted by $\trace(A)$, and the determinant is denoted by $\det(A)$. For two
square matrices $A$ and $B$, $\froinnerprod AB\define\sum_{ij}A_{ij}B_{ij}$
denotes the Frobenius inner product between them, and the induced
norm is denoted by $\fronorm A\define\froinnerprod[1/2]AA$.

The matrix
exponential map of a real matrix is defined by $\exp(A)=\sum_{k=0}^{\infty}A^{k}/k!$,
and its inverse, the matrix logarithm, whenever it exists and is real, is denoted
by $\log(A)$. It is noted that the exponential of a lower triangular matrix is also lower triangular. In addition, the matrix exponential of a diagonal matrix can be obtained by applying the exponential function to each diagonal element. The matrix logarithm of a diagonal matrix with positive diagonal elements can be computed in a similar way. Thus, the matrix exponential/logarithmic map of a diagonal matrix is diagonal.

The space of $m\times m$ lower triangular matrices is
denoted by $\lowtri$, and the subset of $\lowtri$ whose diagonal elements are all positive is denoted by $\cholspace$. It is straightforward to check the following properties of lower triangular matrices.
\begin{itemize}
\item $X=\lp X+\dg X$ for $X\in\lowtri$.
\item $X_1+X_2\in \lowtri$ and $X_1X_2\in \lowtri$ for $X_1,X_2\in\lowtri$.
\item $L_1+L_2\in \cholspace$ and $L_1L_2\in\cholspace$ if $L_1,L_2\in\cholspace$.
\item If $L\in\cholspace$, then the inverse $L^{-1}$ exists and belongs to $\cholspace$.
\item For $X_1,X_2\in\lowtri$, $\dg{X_1+X_2}=\dg{X_1}+\dg{X_2}$ and 
$\dg{X_1X_2}=\dg{X_1}\dg{X_2}$.
\item $\dg{L^{-1}}=\dg{L}^{-1}$ for $L\in \cholspace$.
\item $\det(X)=\prod_{j=1}^m X_{jj}$ for $X\in \lowtri$.
\end{itemize}
These properties show that both $\lowtri$ and $\cholspace$ are closed under matrix addition and multiplication, and that the operator $\dgmap$ interacts well with these operations.

Recall that $\spd$ is defined as the collection of $m\times m$ SPD matrices. We denote the space of $m\times m$ symmetric space 
by $\sym$. Symmetric matrices and SPD matrices possess numerous algebraic and analytic properties that are well documented in \cite{Bhatia2007}. Below are some of them to be used in the sequel.
\begin{itemize}
\item All eigenvalues $\lambda_1,\ldots,\lambda_m$ of an SPD $P$ are positive, and $\det(P)=\prod_{j=1}^m\lambda_j$. Therefore, the determinant of an SPD matrix is positive.
\item For any invertible matrix $X$, the matrix $XX^\transpose$ is an SPD matrix.
\item $\exp(S)$ is an SPD matrix for a symmetric matrix $S$, while $\log(P)$ is a symmetric metric for an SPD matrix $P$.
\item Diagonal elements of an SPD matrix are all positive. This can be seen from the fact that $P_{jj}=e_j^\transpose P e_j>0$ for $P\in\spd$, where $e_j$ is the unit vector with 1 at the $j$th coordinate and 0 elsewhere.
\end{itemize}

\subsection{Cholesky Decomposition}
\label{subsec:cholesky-decomposition}
Cholesky decomposition, named after Andr\'e-Louis
Cholesky, represents a real $m\times m$ SPD
matrix $P$ as a product of a lower triangular matrix $L$ and
its transpose, i.e., $P=LL^{\transpose}$. If the diagonal elements of
$L$ are restricted to be positive, then the decomposition is unique
according to Theorem 4.2.5 of \cite{Golub1996}. Such lower
triangular matrix, denoted by $\ch P$, is called the \emph{Cholesky factor}
of $P$. Since in addition $L=\ch{LL^{\transpose}}$ for each $L\in\cholspace$,
the map $\chmap:\spd\rightarrow\cholspace$ is bijective. In other
words, there is one-to-one correspondence between SPD matrices and
lower triangular matrices whose diagonal elements are all positive.

The space $\cholspace$, called the \emph{Cholesky space} in this paper,
is a smooth submanifold of $\lowtri$ that is identified with the
Euclidean space $\real^{m(m+1)/2}$. Similarly, the space $\spd$
of SPD matrices is a smooth submanifold of
the space $\sym$ of symmetric matrices identified with vectors in
$\real^{m(m+1)/2}$. As a manifold map between smooth manifolds $\cholspace$ and $\spd$, the map $\chmap$ is indeed a diffeomorphism. This fact will be explored to endow $\spd$ with a new Riemannian metric that to be presented in \cref{subsec:spd-metric}.
\begin{prop}
The Cholesky map $\chmap$ is a diffeomorphism between smooth manifolds
$\cholspace$ and $\spd$.
\end{prop}

\begin{proof}
We have argued that $\chmap$ is a bijection. To see that it is also
smooth, for $P=LL^{\transpose}$ with $L\in\cholspace$, we write
\begin{align*}
& \begin{pmatrix}P_{11} & P_{12} & \cdots & P_{1m}\\
P_{21} & P_{22} & \cdots & P_{2m}\\
\vdots & \vdots & \ddots & \vdots\\
P_{m1} & P_{m2} & \cdots & P_{mm}
\end{pmatrix}  =\begin{pmatrix}L_{11} & 0 & \cdots & 0\\
L_{21} & L_{22} & \cdots & 0\\
\vdots & \vdots & \ddots & \vdots\\
L_{m1} & L_{m2} & \cdots & L_{mm}
\end{pmatrix}\begin{pmatrix}L_{11} & L_{21} & \cdots & L_{m1}\\
0 & L_{22} & \cdots & L_{m2}\\
\vdots & \vdots & \ddots & \vdots\\
0 & 0 & \cdots & L_{mm}
\end{pmatrix}\\
 & =\begin{pmatrix}L_{11}^{2} & L_{21}L_{11} & \cdots & L_{m1}L_{11}\\
L_{21}L_{11} & L_{21}^{2}+L_{22}^{2} & \cdots & L_{m1}L_{21}+L_{m2}L_{22}\\
\vdots & \vdots & \ddots & \vdots\\
L_{m1}L_{11} & L_{m1}L_{21}+L_{m2}L_{22} & \cdots & \sum_{k=1}^{m}L_{mk}^{2}
\end{pmatrix},
\end{align*}
from which we deduce that 
\begin{equation}
\begin{cases}
L_{ii}=\sqrt{P_{ii}-\sum_{k=1}^{i-1}L_{ik}^{2}},\\
L_{ij}=\frac{1}{L_{jj}}\left(P_{ij}-\sum_{k=1}^{j-1}L_{ik}L_{jk}\right) & \text{for }i>j.
\end{cases}\label{eq:pf-chol-decomp}
\end{equation}
The existence of a unique Cholesky factor for every SPD matrix suggests
that $P_{ii}-\sum_{k=1}^{i-1}L_{ik}^{2}>0$ for all $i$. Thus, $L_{11}=\sqrt{P_{11}}$,
as well as its reciprocal $1/L_{11}$, is smooth. 

Now assume $L_{ij}$ and $1/L_{jj}$
are smooth for $i=1,\ldots,i_{0}$ and $j=1,\ldots,j_{0}\leq i_{0}$.
As we just showed, this hypothesis is true for $i_{0}=1$ and $j_{0}=1$.
If $j_{0}=i_{0}$, from \cref{eq:pf-chol-decomp} we see that $L_{i_{0}+1,1}=(1/L_{11})P_{i_{0}+1,1}$
is smooth. If $j_{0}<i_{0}-1$, then $L_{i_{0},j_{0}+1}$ results
from a sequence of elementary operations, such as multiplication,
addition and subtraction, of maps $L_{i_{0}1},\ldots,L_{i_{0},j_{0}}$,
$L_{j_{0}+1,1},\ldots,L_{j_{0}+1,j_{0}}$ and $1/L_{j_{0}+1,j_{0}+1}$
that are all smooth according to the induction hypothesis. As these elementary
operations are all smooth, $L_{i_{0},j_{0}+1}$ is also smooth. If
$j_{0}=i_{0}-1$, then $L_{i_{0},j_{0}+1}=L_{i_{0},i_{0}}$, as well
as $1/L_{i_{0},i_{0}}$, is smooth via similar reasoning based on
the additional fact that $P_{i_{0},i_{0}}-\sum_{k=1}^{i_{0}-1}L_{ik}^{2}>0$
and the square-root operator $\sqrt{}$ is smooth on the set of positive real
numbers. The above derivation then shows that the induction hypothesis
is also true for $i=i_{0},j=j_{0}+1$ if $j_{0}<i_{0}$ and $i=i_{0}+1,j=1$
if $j_{0}=i_{0}$. 

Consequently, by mathematical induction, the hypothesis
is true for all pairs of $i$ and $j\leq i$. In other words, $\chmap$
is a smooth manifold map. Its inverse, denoted by $\chtospd$, is
given by $\chtospd(L)=LL^{\transpose}$ and clearly smooth. Therefore,
$\chmap$ and its inverse $\chtospd$ are diffeomorphisms.
\end{proof}

\section{Lie group structure and bi-invariant metric\label{sec:Lie-group-structure}}
In this section, we first construct a group structure and a bi-invariant metric on the manifold $\cholspace$, and then push them forward to the manifold $\spd$ via the Cholesky map. Parallel transport on SPD manifolds is also investigated. For a background of Riemannian geometry and Lie group, we recommend monographs \cite{Helgason2001,Lang1995}.

\subsection{Riemannian geometry on Cholesky spaces}

For matrices in $\cholspace$,
as off-diagonal elements in the lower triangular part are unconstrained
while  diagonal ones are restricted to be positive, $\cholspace$
can be parameterized by $\lowtri\ni X:\rightarrow\varphi(X)\in\cholspace$
in the way that $(\varphi(X))_{ij}=X_{ij}$ if $i\neq j$ and $(\varphi(X))_{jj}=\exp(X_{jj})$.
This motivates us to respectively endow the unconstrained part $\lp{\cdot}$ and the positive diagonal part $\dg{\cdot}$ with
a different metric and then combine them into a Riemannian metric
on $\cholspace$, as follows. First, we note that the tangent space
of $\cholspace$ at a given $L\in\cholspace$ is identified with the
linear space $\lowtri$. For such tangent space, we treat the strict lower triangular
space $\lp{\lowtri}\define\{\lp X:X\in\lowtri\}$ as the Euclidean
space $\real^{m(m-1)/2}$ with the usual Frobenius inner product $\froinnerprod XY=\sum_{i,j=1}^{m}X_{ij}Y_{ij}$
for all $X,Y\in\lp{\lowtri}$. For the diagonal part $\dg{\lowtri}\define\{\dg Z:Z\in\lowtri\}$,
we equipped it with a different inner product defined by $\froinnerprod{\dg L^{-1}\dg X}{\dg L^{-1}\dg Y}$.
Finally, combining these two components together, we define a metric
$\tilde{g}$ for tangent spaces $\tangentspace[\cholspace]L$ (identified
with $\lowtri)$ by 
\begin{align*}
\gchol_{L}(X,Y) & =\froinnerprod{\lp X}{\lp Y}+\froinnerprod{\dg L^{-1}\dg X}{\dg L^{-1}\dg Y}\\
& =\sum_{i>j}X_{ij}Y_{ij}+\sum_{j=1}^{m}X_{jj}Y_{jj}L_{jj}^{-2}.
\end{align*}
It is straightforward to show that the space $\cholspace$, equipped
with the metric $\gchol$, is a Riemannian manifold. We begin with geodesics on the manifold $(\cholspace,\gchol)$ to investigate
its basic properties.
\begin{prop}
\label{prop:geodesic-chol}On the Riemannian manifold $(\cholspace,\gchol)$,
the geodesic starting at $L\in\cholspace$ with direction $X\in\tangentspace[\cholspace]L$
is given by $$\geochol_{L,X}(t)=\lp L+t\lp X+\dg L\exp\{t\dg X\dg L^{-1}\}.$$ 
\end{prop}

\begin{proof}
Clearly, $\geochol_{L,X}(0)=L$ and $\geochol_{L,X}^{\prime}(0)=\lp X+\dg X=X$.
Now, we use $\mathrm{vec}(L)$ to denote the vector in $\real^{m(m+1)/2}$ such
that the first $m$ elements of $\mathrm{vec}(L)$ correspond to the diagonal elements of $L$. Define the map $x:\cholspace\rightarrow\real$
by
\[
x^{i}(L)=\begin{cases}
\log \mathrm{vec}(L)_{i} & \text{if }1\leq i\leq m,\\
\mathrm{vec}(L)_{i} & \text{otherwise},
\end{cases}
\]
where $x^{i}$ denotes the $i$th component of $x$. It can be checked
that $(\cholspace,x)$ is a chart for the manifold $\cholspace$.
Let $e_{i}$ be the $m(m+1)/2$ dimensional vector whose $i$th element
is one and other elements are all zero. For $1\leq i\leq m$, we define
$\partial_{i}=\mathrm{vec}(L)_ie_{i}$, and for $i>m$, define $\partial_{i}=e_{i}$.
The collection $\{\partial_{1},\ldots,\partial_{m(m+1)/2}\}$ is a frame. One can check that $\gchol_{ij}\define\gchol_{L}(\partial_{i},\partial_{j})=0$
if $i\neq j$, and $\gchol_{ii}=1$. This implies that $\partial\gchol_{jk}/\partial x^{l}=0$,
and hence all Christoffel symbols are all zeros, as
\[
\Gamma_{\:kl}^{i}=\frac{1}{2}\gchol^{ij}\left(\frac{\partial\gchol_{jk}}{\partial x^{l}}+\frac{\partial\gchol_{jl}}{\partial x^{k}}-\frac{\partial\gchol_{kl}}{\partial x^{j}}\right)=0,
\]
where Einstein summation convention is assumed. It can be checked that the $i$th coordinate $\geochol^{i}(t)=x^{i}\circ\geochol_{L,X}(t)$
of the curve $\tilde{\gamma}_{L,X}$ is given by $\geochol^{i}(t)=\log \mathrm{vec}(L)_{i}+t\mathrm{vec}(X)_{i}/\mathrm{vec}(L)_{i}$
when $i\leq m$ and $\geochol^{i}(t)=\mathrm{vec}(L)_{i}+t\mathrm{vec}(X)_{i}$ if $i>m$. Now, it
is an easy task to verify the following geodesic equations
\[
\frac{d^{2}\geochol^{i}}{dt^{2}}+\Gamma_{jk}^{i}\frac{d\geochol^{j}}{dt}\frac{d\geochol^{k}}{dt}=0
\]
for $i=1,\ldots,m(m+1)/2$. Therefore, $\geochol_{L,X}(t)$ is
the claimed geodesic.
\end{proof}

Given the above proposition, we can immediately derive the Riemannian
exponential map $\rechol$ at $L\in\cholspace$, which is given by
$$\rechol_{L}X=\geochol_{L,X}(1)=\lp L+\lp X+\dg L\exp\{\dg X\dg L^{-1}\}.$$
Also, for $L,K\in\cholspace$, with $X=\lp K-\lp L+\{\log\dg K-\log\dg L\}\dg L$,
one has $$\geochol_{L,X}(t)=\lp L+t\{\lp K-\lp L\}+\dg L\exp[t\{\log\dg K-\log\dg L\}].$$
Since $\geochol_{L,X}(0)=L$ and $\geochol_{L,X}(1)=K$, $\geochol_{L,X}$
is the geodesic connecting $L$ and $K$. Therefore, the distance
function on $\cholspace$ induced by $\gchol$, denoted by $d_{\cholspace}$,
is given by 
\[
d_{\cholspace}(L,K)=\{\gchol_{L}(X,X)\}^{1/2}=\left\{ \sum_{i>j}(L_{ij}-K_{ij})^{2}+\sum_{j=1}^{m}(\log L_{jj}-\log K_{jj})^{2}\right\} ^{1/2},
\]
where $X$ is the same as the above. The expression for the distance
function can be equivalently and more compactly written as $$d_{\cholspace}(L,K)=\{\fronorm[2]{\lp L-\lp K}+\fronorm[2]{\log\dg L-\log\dg K}\}^{1/2}.$$
\Cref{tab:mfd-chol} summarizes the above basic properties of
the manifold $(\cholspace,\gchol)$. 
\begin{table}
\renewcommand{\arraystretch}{1.5}
\caption{Basic properties of Riemannian manifolds $(\protect\cholspace,\protect\gchol)$. }
\begin{centering}
\begin{tabular}{|c|}
\hline 
tangent space at $L$ \tabularnewline
 $\lowtri$\tabularnewline
\hline 
\hline 
Riemannian metric \tabularnewline
 $\gchol_{L}(X,Y)=\sum_{i>j}X_{ij}Y_{ij}+\sum_{j=1}^{m}X_{jj}Y_{jj}L_{jj}^{-2}$\tabularnewline
\hline 
\hline 
geodesic emanating from $L$ with direction $X$ \tabularnewline
 $\geochol_{L,X}(t)=\lp L+t\lp X+\dg L\exp\{t\dg X\dg L^{-1}\}$\tabularnewline
\hline 
\hline 
Riemannian exponential map at $L$ \tabularnewline
 $\rechol_{L}X=\lp L+\lp X+\dg L\exp\{\dg X\dg L^{-1}\}$\tabularnewline
\hline 
\hline 
Riemannian logarithmic map at $L$ \tabularnewline
 $\logchol_{L}K=\lp K-\lp L+\dg L\log\{\dg L^{-1}\dg K\}$\tabularnewline
\hline 
\hline 
geodesic distance between $L$ and $K$ \tabularnewline
 $d_{\cholspace}(L,K)=\{\fronorm[2]{\lp L-\lp K}+\fronorm[2]{\log\dg L-\log\dg K}\}^{1/2}$\tabularnewline
\hline 
\end{tabular}
\par\end{centering}
\label{tab:mfd-chol}
\end{table}

\subsection{Riemannian metric for SPD matrices}\label{subsec:spd-metric}
As mentioned previously, the space $\spd$ of SPD matrices is a smooth submanifold of the space $\sym$ of symmetric matrices, whose tangent space at a given SPD matrix is identified with $\sym$. We also showed that  the map $\chtospd:\cholspace\rightarrow\spd$ by $\chtospd(L)=LL^{\transpose}$
is a diffeomorphism between $\cholspace$ and $\spd$. For a square matrix $S$, we define a lower
triangular matrix $S_{\half}=\lp S+\dg S/2$. In another word, the
matrix $S_{\half}$ is the lower triangular part of $S$ with the
diagonal elements halved. The differential of $\chtospd$ is given in the following.
\begin{prop}
The differential $\differential{\chtospd}L:\tangentspace[\cholspace]L\rightarrow\tangentspace[{\spd}]{LL^{\transpose}}$
of $\chtospd$ at $L$ is given by $$(\differential{\chtospd}L)(X)=LX^{\transpose}+XL^{\transpose}.$$
Also, the inverse $(\differential{\chtospd}L)^{-1}:\tangentspace[{\spd}]{LL^{\transpose}}\rightarrow\tangentspace[\cholspace]L$
of $\differential{\chtospd}L$ exists for all $L\in\cholspace$ and
is given by $$(\differential{\chtospd}L)^{-1}(W)=L(L^{-1}WL^{-\transpose})_{\half}$$
for $W\in\sym$.
\end{prop}
\begin{proof}
Let $X\in\lowtri$ and $L\in\cholspace$. Then $\geochol_{L}(t)=L+tX$
is a curve passing through $L$ if $t\in(-\epsilon,\epsilon)$ for
a sufficiently small $\epsilon>0$. Note that for every such $t$,
$\geochol_{L}(t)\in\cholspace$. 
Then $\geospd_{LL^{\transpose}}(t)=\chtospd(\geochol_{L}(t))$ is
a curve passing through $LL^{\transpose}$. The differential is then
derived from $$(\differential{\chtospd}L)(X)=(\chtospd\circ\geochol_{L})^{\prime}(0)=\frac{\diffop}{\diffop t}\chtospd(\geochol_{L}(t))\mid_{t=0}=LX^{\transpose}+XL^{\transpose}.$$ 

On the other hand, if $W=LX^{\transpose}+XL^{\transpose}$, then since
$L$ is invertible, we have $$L^{-1}WL^{-\transpose}=X^{\transpose}L^{-\transpose}+L^{-1}X=L^{-1}X+(L^{-1}X)^{\transpose}.$$
Note that $L^{-1}X$ is also a lower triangular matrix, and the matrix
on the left hand side is symmetric, we deduce that $(L^{-1}WL^{-\transpose})_{\half}=L^{-1}X$,
which gives rise to $X=L(L^{-1}WL^{-\transpose})_{\half}$. The linear
map $(\differential{\chtospd}L)(X)=LX^{\transpose}+XL^{\transpose}$
is one-to-one, since from the above derivation, $(\differential{\chtospd}L)(X)=0$
if and only if $X=0$.
\end{proof}
Given the above proposition, the manifold map $\chtospd$, which is exactly the inverse of the Cholesky map $\chmap$ discussed in \cref{subsec:cholesky-decomposition}, induces
a Riemannian metric on $\spd$, denoted by $\gspd$ and called \emph{Log-Cholesky
metric}, given by 
\begin{equation}
\gspd_{P}(W,V)=\gchol_{L}\left(L(L^{-1}WL^{-\transpose})_{\half},L(L^{-1}VL^{-\transpose})_{\half}\right),\label{eq:spd-metric}
\end{equation}
where $L=\chtospd^{-1}(P)=\ch P$ is the Cholesky factor of $P\in\spd$,
and $W,V\in\sym$ are tangent vectors at $P$. \linremark{This implies that $$\gchol_L(X,Y)=\gspd_{\chtospd(L)}\big((D_L\chtospd)(X),(D_L\chtospd)(Y)\big)$$ for all $L$ and $X,Y\in\tangentspace[{\cholspace}]{L}$.  According to Definition 7.57 of \cite{Lee2009}, the map $\chtospd$
is an isometry between $(\cholspace,\gchol)$ and $(\spd,\gspd)$. A Riemannian isometry provides correspondence of Riemannian properties and objects between two Riemannian manifolds. This enables us to study the properties of $(\spd,\gspd)$ via the manifold $(\cholspace,\gchol)$ and the isometry $\chtospd$. For example, we can obtain geodesics on $\spd$ by mapping geodesics on $\cholspace$. More precisely,}
 the geodesic emanating from $P=LL^{\transpose}$ with $L=\ch P$
is given by
\[
\geospd_{P,W}(t)=\chtospd\left(\geochol_{L,X}(t)\right)=\geochol_{L,X}(t)\geochol_{L,X}^{\transpose}(t),
\]
where $X=L(L^{-1}WL^{-\transpose})_{\half}$ and $W\in\tangentspace[\spd]{P}$. Similarly, the Riemannian exponential at $P$ 
is given by $$\respd_{P}W=\chtospd(\rechol_{L}X)=(\rechol_{L}X)(\rechol_{L}X)^{\transpose},$$
while the geodesic between $P$ and $Q$ is characterized by 
\[
\geospd_{P,W}(t)=\geochol_{L,X}(t)\geochol_{L,X}^{\transpose}(t),
\]
with $L=\ch P$, $K=\ch Q$, $X=\lp K-\lp L+\{\log\dg K-\log\dg L\}\dg L$
and $W=LX^{\transpose}+XL^{\transpose}$. Also, the geodesic distance
between $P$ and $Q$ is $$d_{\spd}(P,Q)=d_{\cholspace}(\ch P,\ch Q).$$
\linremark{Moreover, the Levi-Civita connection $\nabla$ of $(\spd,\gspd)$ can be obtained by the Levi-Civita connection $\tilde{\nabla}$ of $(\cholspace,\gchol)$. To see this, let $W$ and $V$ be two smooth vector fields on $\spd$. Define vector fields $X$ and $Y$ on $\cholspace$ by $X(L)=(D_{LL^\transpose}\chmap)W(LL^\transpose)$ and $Y(L)=(D_{LL^\transpose}\chmap)V(LL^\transpose)$ for all $L\in\cholspace$. Then $\nabla_WV=(D\chtospd)(\tilde{\nabla}_XY)$, and the Christoffel symbols to compute the connection $\tilde{\nabla}$ has been given  in the proof of  \cref{prop:geodesic-chol}.}

\Cref{tab:mfd-spd} summarizes some basic properties of the manifold
$(\spd,\gspd)$. Note that the differential $\differential{\chmap}P$
can be computed efficiently, since it only involves Cholesky decomposition
and the inverse of a lower triangular matrix, for both of which there
exist efficient algorithms. Consequently, all  maps
in \cref{tab:mfd-spd} can be evaluated in an efficient way.
In contrast, computation of Riemannian exponential/logarithmic maps
for the Log-Euclidean metric \cite{Arsigny2007} requires evaluation of 
some series of an infinite number of terms; see Eq. (2.1) and Table
4.1 of \cite{Arsigny2007}.

\begin{table}
\renewcommand{\arraystretch}{1.5}
\caption{Properties of Riemannian manifold $(\protect\spd,\protect\gspd)$. }
\begin{centering}
\begin{tabular}{|c|}
\hline 
tangent space at $P$ \tabularnewline
 $\sym$\tabularnewline
\hline 
\hline 
differential of $\chtospd$ at $L$ \tabularnewline
 $\differential{\chtospd}L:\,X\longmapsto LX^{\transpose}+XL^{\transpose}$\tabularnewline
\hline 
\hline 
differential of $\chmap$ at $P$ \tabularnewline
 $\differential{\chmap}P:\,W\longmapsto\ch P(\ch P^{-1}W\ch P^{-\transpose})_{\frac{1}{2}}$ \tabularnewline
\hline 
\hline 
Riemannian metric  \tabularnewline
 $\gspd_{P}(W,V)=\gchol_{\ch P}((\differential{\chmap}P)(W),(\differential{\chmap}P)(W))$\tabularnewline
\hline 
\hline 
geodesic emanating from $P$ with direction $W$ \tabularnewline
 $\geospd_{P,W}(t)=\geochol_{\ch P,(\differential{\chmap}P)(W)}(t)\geochol_{\ch P,(\differential{\chmap}P)(W)}^{\transpose}(t)$\tabularnewline
\hline 
\hline 
Riemannian exponential map at $P$ \tabularnewline
 $\respd_{P}W=\rechol_{\ch P}(\differential{\chmap}P)(W))\{\rechol_{\ch P}(\differential{\chmap}P)(W)\}^{\transpose}$\tabularnewline
\hline 
\hline 
Riemannian logarithmic map at $P$ \tabularnewline
 $\Log_{P}Q=(\differential{\chtospd}{\ch P})(\logchol_{\ch P}\ch Q)$\tabularnewline
\hline 
\hline 
geodesic distance between $P$ and $Q$ \tabularnewline
 $d_{\spd}(P,Q)=d_{\cholspace}(\ch P,\ch Q)$.\tabularnewline
\hline 
\end{tabular}
\par\end{centering}
\label{tab:mfd-spd}
\end{table}

\subsection{Lie group structure and bi-invariant metrics }

We define an operator $\mygroupop$ on $\lowtri$ by $$X\mygroupop Y=\lp X+\lp Y+\dg X\dg Y.$$
Note that $\cholspace\subset\lowtri$. Moreover, if $L,K\in\cholspace$,
then $L\mygroupop K\in\cholspace$. It is not difficult to see that
$\mygroupop$ is a smooth commutative group operation on the manifold $\cholspace$, where the inverse of $L$, denoted by $L_{\mygroupop}^{-1}$, is $\dg{L}^{-1}-\lp{L}$.
The left translation by $A\in\cholspace$ is denoted by $\ell_{A}:B\mapsto A\mygroupop B$.
One can check that the differential of this operation \linremark{at $L\in\cholspace$} is 
\begin{equation}\label{eq:diff-lA}
\differential{\ell_{A}}L:X\mapsto\lp X+\dg A\dg X,
\end{equation}
\linremark{where it is noted that the differential $\differential{\ell_A}L$ does not depend on $L$. Given the above expression, one can find that}
\begin{align*}
& \gchol_{A\mygroupop L}((\differential{\ell_{A}}L)(X),(\differential{\ell_{A}}L)(Y)) \\
& =\gchol_{A\mygroupop L}(\lp X+\dg A\dg X,\lp Y+\dg A\dg Y) \\
& =\gchol_{L}(X,Y).
\end{align*}
Similar observations are made for right translations. Thus, the metric
$\gchol$ is a bi-invariant metric that turns $(\cholspace,\mygroupop)$ into a Lie group. 

The group operator $\mygroupop$ and maps $\chtospd$ and $\chmap$ together induce
a smooth operation $\spdgpop$ on $\spd[m]$, defined by
\begin{align*}
 P\spdgpop Q & =\chtospd(\ch P\mygroupop\ch Q)\\
 & =(\ch P\mygroupop\ch Q)(\ch P\mygroupop\ch Q)^{\transpose},\quad\text{for }P,Q\in\spd.
\end{align*}
In addition, both $\chmap$ and $\chtospd$ are Riemannian isometries and group
isomorphisms between Lie groups $(\cholspace,\gchol,\mygroupop)$ and $(\spd,\gspd,\spdgpop)$. 

\begin{thm}
The space $(\spd,\spdgpop)$ is an abelian Lie group. Moreover, the
metric $\gspd$ defined in \cref{eq:spd-metric} is a bi-invariant
metric on $(\spd,\spdgpop)$.
\end{thm}

\begin{proof}
It is clear that $\spd$ is closed under the operation $\spdgpop$, and the identity element is the identity matrix.
For $P\in\spd$, the inverse under $\spdgpop$ is given by $(\ch P_{\mygroupop}^{-1})(\ch P_{\mygroupop}^{-1})^{\transpose}$. For associativity, we
first observe that $\ch{P\spdgpop Q}=\ch P\mygroupop\ch Q$, based
on which we further deduce that
\begin{align*}
(P\spdgpop Q)\spdgpop S & =(\ch{P\spdgpop Q}\mygroupop\ch S)(\ch{P\spdgpop Q}\mygroupop\ch S)^{\transpose}\\
 & =(\ch P\mygroupop\ch Q\mygroupop\ch S)(\ch P\mygroupop\ch Q\mygroupop\ch S)^{\transpose}\\
 & =(\ch P\mygroupop\ch{Q\spdgpop S})(\ch P\mygroupop\ch{Q\spdgpop S})^{\transpose}\\
 & =P\spdgpop(Q\spdgpop S).
\end{align*}
Therefore, $(\spd,\spdgpop)$ is a group. The commutativity and smoothness
of $\spdgpop$ stem from the commutativity and smoothness of $\mygroupop$,
respectively. It can be checked that $\chtospd$ is a group isomorphism
and isometry between Lie groups $(\cholspace,\mygroupop)$ and $(\spd,\spdgpop)$
respectively endowed with Riemannian metrics $\gchol$ and $\gspd$.
Then, the bi-invariance of $\gspd$ follows from the bi-invariance
of $\gchol$. 
\end{proof}

\subsection{Parallel transport along geodesics on $\protect\spd$}

In some applications like statistical analysis or machine learning
on Riemannian manifolds, parallel transport of tangent vectors along
geodesics is required. For instance, in \cite{Yuan2012} that studies
regression on SPD-valued data, tangent vectors are transported to
a common place to derive statistical estimators of interest. Also,
optimization in the context of statistics for manifold-valued data
often involves parallel transport of tangent vectors. Examples include
 Hamiltonian Monte Carlo algorithms  \cite{Kim2015}
as well as  optimization algorithms  \cite{Hosseini2015}
to train manifold-valued Gaussian mixture models. In these scenarios,
Riemannian metrics on SPD matrices that result in efficient computation
for parallel transport are attractive, in particular in the era
of big data. In this regard, as discussed in the introduction, evaluation
of parallel transport along geodesics for  the affine-invariant metric
is simple and efficient in computation, while the one for the Log-Euclidean
metric is computationally intensive. Below we show that parallel transport
for the presented metric also has a simple form, starting with a lemma.
\begin{lem}
\label{lem:transport}Let $(\mathcal{G},\cdot)$ be an abelian Lie
group with a bi-invariant metric. The parallel transporter $\tau_{p,q}$
that transports tangent vectors at $p$ to tangent vectors at $q$
along geodesics connecting $p$ and $q$ is given by $\tau_{p,q}(u)=(\differential{\ell_{q\cdot p^{-1}}}p)u$
for $u\in T_{p}\mathcal{G}$.
\end{lem}

\begin{proof}
For simplicity, we abbreviate $p\cdot q$ as $pq$. Let $\mathfrak{g}$
denote the Lie algebra associated with the Lie group $\mathcal{G}$,
and $\nabla$ the Levi-Civita connection on $\mathcal{G}$. Note that
we identify elements in $\mathfrak{g}$ with left-invariant vector
fields on $\mathcal{G}$. We shall first recall that  $\nabla_{Y}Z=[Y,Z]/2$ for $Y,Z\in\mathfrak{g}$ (see the proof of Theorem 21.3 in \cite{Milnor1963}),
where $[\cdot,\cdot]$ denotes the Lie bracket of $\mathcal{G}$.
As $\mathcal{G}$ is abelian, the Lie bracket vanishes everywhere and hence $\nabla_{Y}Z=0$ if $Y,Z\in\mathfrak{g}$.

Let $\gamma_{p}(t)=\ell_{p}(\mathfrak{exp}(tY))$ for $Y\in\mathfrak{g}$
such that $\mathfrak{exp}(Y)=p^{-1}q$, where $\mathfrak{exp}$ denotes
the Lie group exponential map. Recall that for a bi-invariant Lie group, the group exponential map coincides with the Riemannian exponential
map at the group identity $e$, and left translations are isometries. Thus, $\gamma_{p}$ is a
geodesic. Using the fact $\gamma_{e}(t+s)=\gamma_{e}(t)\gamma_{e}(s)=\ell_{\gamma_{e}(t)}(\gamma_e(s))$
according to Lemma 21.2 of \cite{Milnor1963}, by the chain rule of differential,
we have 
\[
\gamma_{e}^{\prime}(t)=\frac{\diffop}{\diffop t}\gamma_{e}(t+s)\mid_{s=0}=(\differential{\ell_{\gamma_{e}(t)}}e)\left(\gamma_e^{\prime}(0)\right)=(\differential{\ell_{\gamma_{e}(t)}}e)\left(Y(e)\right)=Y(\gamma_{e}(t)),
\]
from which we further deduce that $$\gamma_{p}^{\prime}(t)=(\differential{\ell_{p}}{\gamma_{e}(t)})\gamma_{e}^{\prime}(t)=(\differential{\ell_{p}}{\gamma_{e}(t)})Y(\gamma_{e}(t))=Y(p\gamma_{e}(t)).$$

Now define a vector field $Z(q)\define(\differential{\ell_{qp^{-1}}}p)u$.
We claim that $Z$ is a left-invariant vector field
on $\mathcal{G}$ and hence belongs to $\mathfrak{g}$, since 
\begin{align*}
Z(hq) & =(\differential{\ell_{hqp^{-1}}}p)u =\{\differential{(\ell_{h}\circ\ell_{qp^{-1}}}p)\}u\\
& =(\differential{\ell_{h})}q(\differential{\ell_{qp^{-1}}}p)u  =(\differential{\ell_{h})}q(Z(q)),
\end{align*}
where the third equality is obtained by the chain rule of differential.
Consequently, $\nabla_{\gamma_{p}^{\prime}}Z=0$ since
$\gamma_{p}^{\prime}(t)=Y(p\gamma_{e}(t))$ and  $\nabla_YZ=0$ for $Y,Z\in\mathfrak{g}$. As additionally $Z(\gamma_{p}(0))=Z(p)=u$,
transportation of $u$ along the geodesic $\gamma_{p}$ is realized
by the left-invariant vector field $Z$. Since $\gamma_{p}$
is a geodesic with $\gamma_{p}(0)=p$ and $\gamma_{p}(1)=\ell_{p}\mathfrak{exp}(Y)=p(p^{-1}q)=q$,
we have that $$\tau_{p,q}(u)=Z(\gamma_{p}(1))=Z(q)=(\differential{\ell_{qp^{-1}}}p)u,$$
as claimed.
\end{proof}
\begin{prop}
\label{prop:transport}A tangent vector $X\in\tangentspace[\cholspace]L$
is parallelly transported to the tangent vector $\lp X+\dg K\dg L^{-1}\dg X$
at $K$. For $P,Q\in\spd$ and $W\in\sym$, 
$$\tau_{P,Q}(W)=K\{\lp X+\dg K\dg L^{-1}\dg X\}^\transpose+\{\lp X+\dg K\dg{L^{-1}}\dg X\}K^{\transpose},$$
where $L=\ch P$, $K=\ch Q$ and $X=L(L^{-1}WL^{-\transpose})_{\frac{1}{2}}$. 
\end{prop}

\begin{proof}
By \cref{lem:transport}, it is seen that $\tau_{L,K}(X)=(\differential{\ell_{K\mygroupop L_{\mygroupop}^{-1}}}L)X$. According to \cref{eq:diff-lA}, $$(\differential{\ell_{K\mygroupop L_{\mygroupop}^{-1}}}L)X=\lp X+\dg{K\mygroupop L_{\mygroupop}^{-1}}\dg X=\lp X+\dg K\dg L^{-1}\dg X.$$
The statement for $P,Q,W$ follows from the fact that $\chtospd$ and $\chmap$ are
an isometries.
\end{proof}
The above proposition shows that parallel transport for the presented
Log-Cholesky metric can be computed rather efficiently. In fact, the
only nontrivial steps in computation are to perform Cholesky decomposition
of two SPD matrices and to inverse a lower triangular matrix, for
both of which there exist efficient algorithms. It is numerically
faster than the affine-invariant metric and Log-Euclidean metric. For instance, for $5\times5$ SPD matrices,
in a MATLAB computational environment, on average it takes 9.3ms (Log-Euclidean),
0.85ms (affine-invariant) and 0.2ms (Log-Cholesky) on an Intel(R)
Core i7-4500U (1.80GHz) to perform a parallel transport. We see that, to do parallel transport, the Log-Euclidean
metric is about 45 times slower than the Log-Cholesky metric, since it has to evaluate the differential of a matrix logarithmic map that
is expressed as a convergent series of infinite terms of products
of matrices. In practice, only a finite leading terms are evaluated.
However, in order to avoid significant loss of precision, a large
number of terms are often needed. For instance, for a precision of
the order of $10^{-12}$, typically approximately 300 leading terms are required. The affine-invariant metric, despite having a simple form for parallel
transport, is still about 4 times slower than our Log-Cholesky metric, partially due to that more matrix inversion and multiplication operations are needed.

\section{Mean of distributions on $\protect\spd$\label{sec:Mean-of-distributions}}
In this section we study the Log-Cholesky mean of a random SPD matrix and the Log-Cholesky average of a finite number of SPD matrices. We first establish the existence and uniqueness of such quantities. A closed and easy-to-compute form for Log-Cholesky averages is also derived. Finally, we show that determinants of Log-Cholesky averages are bounded by determinants of SPD matrices being averaged. This property suggests that Log-Cholesky averages are not subject to swelling effect.

\subsection{Log-Cholesky mean of a random SPD matrix}
For a random element $Q$ on a Riemannian manifold $\manifold$, we
define a function $F(x)=\expect d_{\manifold}^{2}(x,Q)$, where $d_{\manifold}$
denotes the geodesic distance function on $\manifold$, and $\expect$ denotes expectation of a random number. If $F(x)<\infty$
for some $x\in\manifold$, 
\[
x=\underset{z\in\manifold}{\arg\min}\,F(z),
\]
and $F(z)\geq F(x)$ for all $z\in\manifold$, then $x$ is called a Fr\'echet
mean of $Q$, denoted by $\expect Q$. In general, Fr\'echet mean
might not exist, and when it exists, it might not be unique; see \cite{Aftab2013}
for conditions of the existence and uniqueness of Fr\'echet means. However, for $\spd$ endowed with Log-Cholesky metric, we claim that, \linremark{
if $F(S)<\infty$ for some $S\in\spd$, then the Fr\'echet mean exists and is unique}. Such Fr\'echet mean is termed \emph{Log-Cholesky mean} in this paper. To prove the claim, we first
notice that, similar to the Log-Euclidean metric \cite{Arsigny2007},
the manifold $(\spd,\gspd)$ is a ``flat'' space.
\begin{prop}
\label{thm:sectional-curvature}The sectional curvature of $(\cholspace,\gchol)$
and $(\spd,\gspd)$ is constantly zero.
\end{prop}

\begin{proof}
For $(\cholspace,\gchol)$, in the proof of \cref{prop:geodesic-chol},
it has been shown that all Christoffel symbols are zero under the
selected coordinate. This implies that the Riemannian curvature tensor
is identically zero and hence so is the sectional curvature. The case
of $(\spd,\gspd)$ follows from the fact that $\chtospd$ is an isometry that preserves sectional curvature.
\end{proof}
\begin{prop}
\label{thm:frechet-mean}If $L$ is a random element on $(\cholspace,\gchol)$
and $\expect d_{\cholspace}^{2}(A,L)<\infty$ for some $A\in\cholspace$.
Then the Fr\'echet mean $\expect L$ exists and is unique. Similarly, the Log-Cholesky mean of a random SPD matrix $P$ exists and is unique if $\expect d_{\spd}^{2}(S,P)<\infty$ for some SPD matrix $S$.
\end{prop}

\begin{proof}
In the case of $(\cholspace,\gchol)$, we define $\psi(L)=\lp L+\log\dg L$.
It can be shown that $\psi$ is a diffeomorphism (and hence also a
homeomorphism) between $\cholspace$ and $\lowtri$. Therefore, $\cholspace$
is simply connected since $\lowtri$ is. In \cref{thm:sectional-curvature}
we have shown that $\cholspace$ has zero sectional curvature. Thus,
the existence and uniqueness of Fr\'echet mean follows from Theorem
2.1 of \cite{Bhattacharya2003}. The case of $(\spd,\gspd)$ follows
from the isometry of $\chtospd$.
\end{proof}

The next result shows that the Log-Cholesky mean of a random SPD matrix is computable from the Fr\'echet mean of its Cholesky factor. Also, it characterizes the Fr\'echet mean of a random Cholesky factor, which is important for us to derive a closed form for the Log-Cholesky average of a finite number of matrices in the next subsection.
\begin{prop}
\label{thm:EL-EP}Suppose $L$ is a random element on $\cholspace$
and $P$ is a random element on $\spd$. Suppose for some fixed $A\in\cholspace$
and $B\in\spd$ such that $\expect d_{\cholspace}^{2}(A,L)<\infty$
and $\expect d_{\spd}^{2}(B,P)<\infty$. Then the Fr\'echet mean
of $L$ is given by 
\begin{equation}
\expect L=\expect\lp{L}+\exp\left\{ \expect\log\dg L\right\} ,\label{eq:E-L}
\end{equation}
and the Log-Cholesky mean of $P$ is given by
\begin{equation}
\expect P=\{\expect\ch P\}\{\expect\ch P\}^{\transpose}.\label{eq:E-S}
\end{equation}
\end{prop}

\begin{proof}
Let $F(R)=\expect d_{\cholspace}^{2}(R,L)$. Then
$$F(R)=\expect\fronorm[2]{\lp R-\lp L}+\expect\fronorm[2]{\log\dg R-\log\dg L}\define F_{1}(\lp{R})+F_{2}(\dg{R}).$$ To
minimize $F$, we can separately minimize $F_{1}$ and $F_{2}$, due to
two reasons: 1) $F_{1}$ involves $\lp R$ only while $F_{2}$ involves
$\dg R$, and 2) $\lp R$ and $\dg R$ are disjoint and can vary independently.
For $F_{1}$, we first note that the condition $\expect d_{\cholspace}^{2}(A,L)<\infty$
ensures that $\expect\fronorm[2]{\lp L}<\infty$ and hence the mean $\expect\lp L$
is well defined. Also, it can be checked that this mean minimizes $F_{1}$. 

For $F_{2}$, we note that 
\begin{align*}
F_{2}(\dg{R}) & = \sum_{j=1}^{m}\expect(\log\dg R-\log\dg L)_{jj}^{2}\\
& =\sum_{j=1}^{m}\expect\{\log R_{jj}-\log L_{jj}\}^{2} \define\sum_{j=1}^{m}F_{2j}(R_{jj}).
\end{align*}
Again, these $m$ components $F_{2j}$ can be optimized separately.
From the condition $\expect d_{\cholspace}^{2}(A,L)<\infty$ we deduce that $\expect(\log L_{jj})^{2}<\infty$.
Thus, $\expect\log L_{jj}$ is well defined and minimizes $F_{2j}(e^{x})$ with respect to $x$. Therefore, $\exp\left\{ \expect\log L_{jj}\right\} $
minimizes $F_{2j}(x)$ for each $j=1,\ldots,m$. In matrix form,
this is equivalent to that $\exp\left\{ \expect\log\dg L\right\} $
minimizes $F_{2}(D)$ when $D$ is constrained to be diagonal. Thus,
combined with the optimizer for $F_{1}$, it establishes \cref{eq:E-L}. Then
\cref{eq:E-S} follows from the fact that $\chtospd(L)=LL^{\transpose}$
is an isometry. 
\end{proof}

Finally, we establish the following useful relation between the determinant of the Log-Cholesky mean and the mean of the logarithmic determinant of a random SPD matrix.
\begin{prop}
\label{prop:det}If the Fr\'echet mean of a random element $Q$ on
$(\cholspace,\gchol)$ or $(\spd,\gspd)$ exists, then $\log\det(\expect Q)=\expect\log(\det Q)$.
\end{prop}
\begin{proof}
For the case that $Q$ is a random element on $(\cholspace,\gchol)$,
we denote it by $L$ and observe that in \cref{eq:E-L} $\expect L$
is a lower triangular matrix and $\exp\left\{ \expect\log\dg L\right\} $
is its diagonal part. For a triangular matrix, its determinant is
the product of diagonal elements. Thus, $$\det L=\prod_{j=1}^{m}L_{jj}=\exp[\trace\{\log\dg L\}],$$
or equivalently, $\log\det L=\trace\{\log\dg L\}$. The above observations
imply that
\begin{align*}
\det\mathbb{E}L & =\exp\left[\trace\left\{ \log\dg{\expect L}\right\} \right]=\exp\left[\trace\left\{ \expect\log\dg L\right\} \right]\\
 & =\exp\left[\expect\trace\left\{ \log\dg L\right\} \right]=\exp\left\{ \expect\log(\det L)\right\} ,
\end{align*}
which proves the case of $(\cholspace,\gchol)$. 

For the case that $Q$ is a random element on $(\spd,\gspd)$, we
denote it by $P$ and observe that $\det(AB)=(\det A)(\det B)$ for
any square matrices $A$ and $B$. Let $L=\ch P$. Then, one has
\begin{align*}
\det\expect P & =\det\{(\expect L)(\expect L)^{\transpose}\}\\
 & =\{\det(\expect L)\}\{\det(\expect L)^{\transpose}\}\\
 & =\exp\left\{ \expect\log(\det L)\right\} \exp\left\{ \expect\log(\det L^{\transpose})\right\} \\
 & =\exp\left\{ \expect(\log\det L+\log\det L^{\transpose})\right\} \\
 & =\exp\left[\expect\log\{(\det L)\det(L^{\transpose})\}\right],\\
 & =\exp\left\{ \expect\log\det(LL^{\transpose})\right\}\\
 & =\exp\left\{ \expect\log\det(P)\right\},
\end{align*}
which establishes the statement for the case of $(\spd,\gspd)$.
\end{proof}

\subsection{Log-Cholesky average of finite SPD matrices}

Let $Q_{1},\ldots,Q_{n}$ be $n$ points on a Riemannian manifold
$\manifold$. The Fr\'echet average of these points, denoted by $\expect_{n}(Q_{1},\ldots,Q_{n})$,
 is defined to be the minimizer of function $F_{n}(x)=\sum_{i=1}^{n}d_{\manifold}^{2}(x,Q_{i})$ if such minimizer exists and is unique.
Clearly, this concept is analogous to the Fr\'echet mean of a random
element discussed in the above. In fact, the set of the elements $Q_{1},\ldots,Q_{n}$, always corresponds to a random element $Q$ with the uniform
distribution on the set $\{Q_{1},\ldots,Q_{n}\}$. With this correspondence,
the Fr\'echet average of $Q_{1},\ldots,Q_{n}$ is simply the Fr\'echet
mean of $Q$, i.e., $\expect_{n}(Q_{1},\ldots,Q_{n})=\expect Q$.
The following result is a corollary of \cref{thm:EL-EP}
in conjunction with the above correspondence.
\begin{cor}
For $L_{1},\ldots,L_{n}\in\cholspace$, one has 
\begin{equation}
\mathbb{E}_{n}(L_{1},\ldots,L_{n})=\frac{1}{n}\sum_{i=1}^{n}\lp{L_{i}}+\exp\left\{ n^{-1}\sum_{i=1}^{n}\log\dg{L_{i}}\right\} ,\label{eq:En-L}
\end{equation}
and for $P_{1},\ldots,P_{n}\in\spd$, one has
\begin{equation}
\mathbb{E}_{n}(P_{1},\ldots,P_{n})=\mathbb{E}_{n}\{\ch{P_{1}},\ldots,\ch{P_{n}}\}\mathbb{E}_{n}\{\ch{P_{1}},\ldots,\ch{P_{n}}\}^{\transpose}.\label{eq:En-S}
\end{equation}
\end{cor}

For a general manifold, Fr\'echet averages often do not admit a closed form. For example, no closed form of affine-invariant averages has been found. Strikingly, Log-Euclidean
averages admit a simple and closed form which is attractive in applications.
The above corollary shows that, Log-Cholesky averages enjoy the same nice property
of their Log-Euclidean counterparts. 
The following is a consequence
of \cref{prop:det} and the aforementioned principle of
correspondence between a set of objects and a random object with the discrete
uniform probability distribution on them.
\begin{cor}\label{cor:det-En}
For $L_{1},\ldots,L_{n}\in\cholspace$ and $P_{1},\ldots,P_{n}\in\spd$,
one has
\begin{equation}
\det\mathbb{E}_{n}(L_{1},\ldots,L_{n})=\exp\left(\frac{1}{n}\sum_{i=1}^{n}\log\det L_{i}\right)\label{eq:det-En-L}
\end{equation}
and
\begin{equation}
\det\mathbb{E}_{n}(P_{1},\ldots,P_{n})=\exp\left(\frac{1}{n}\sum_{i=1}^{n}\log\det P_{i}\right).\label{eq:det-En-S}
\end{equation}
\end{cor}

The equation \cref{eq:det-En-S} shows that the determinant of the Log-Cholesky average of $n$ SPD matrices is the geometric mean of their determinants. \linremark{Consequently, the Log-Cholesky average can be considered as a generalization of the geometric mean of SPD matrices \cite{Ando2004}, since according to \cite{Arsigny2007}, this property is ``the common property that should have all generalizations of the geometric mean to SPD matrices''. Note that the Log-Cholesky average is also a Fr\'echet mean which is a generalization of the Euclidean mean and has applications in statistics on Riemannian manifolds \cite{Pennec2006b}, while the geometric mean \cite{Ando2004} is algebraically  constructed and not directly related to Riemannian geometry.} 

\Cref{cor:det-En} also suggests that the determinant of the Log-Cholesky average is equal to the determinant of  the Log-Euclidean and affine-invariant averages.  Thus, like these two averages, the Log-Cholesky average does not suffer from swelling effect. In fact, as a consequence of the above corollary, one has that 
\begin{equation}
\inf_{1\leq i\leq n}\det L_{i}\leq\det\mathbb{E}_{n}(L_{1},\ldots,L_{n})\leq\sup_{1\leq i\leq n}\det L_{i}\label{eq:En-range-L}
\end{equation}
and 
\begin{equation}
\inf_{1\leq i\leq n}\det P_{i}\leq\det\mathbb{E}_{n}(P_{1},\ldots,P_{n})\leq\sup_{1\leq i\leq n}\det P_{i}.\label{eq:En-range-S}
\end{equation}
To see so, we first observe that,
for positive numbers $x_{1},\ldots,x_{n}$, one has $$\inf_{1\leq i\leq m}\log x_{i}\leq n^{-1}\sum_{i=1}^{n}\log x_{i}\leq\sup_{1\leq i\leq m}\log x_{i}.$$
Then, the monotonicity of $\exp(x)$ implies that $$\inf_{1\leq i\leq m}x_{i}\leq\exp\left(n^{-1}\sum_{i=1}^{n}\log x_{i}\right)\leq\sup_{1\leq i\leq m}x_{i},$$
and both \cref{eq:En-range-L} and \cref{eq:En-range-S} follow from \cref{cor:det-En}.

As argued in \cite{Arsigny2007}, proper interpolation of SPD matrices
is of importance in diffusion tensor imaging. The analogy of linear
interpolation for Riemannian manifolds is geodesic interpolation.
Specifically, if $P$ and $Q$ are two points on a manifold and $\gamma(t)$
is a geodesic connecting them such that $\gamma(0)=P$ and $\gamma(1)=Q$,
then we say that $\gamma$ geodesically interpolates $P$ and $Q$.
This notion of linear interpolation via geodesics can be straightforwardly
generalized to bilinear or higher dimensional linear interpolation
of points on a manifold. In \cref{fig:swelling-effect}, we
present an illustration of such geodesic interpolation for SPD matrices
under the Euclidean metric,  Cholesky distance \cite{Dryden2009}, affine-invariant metric, Log-Euclidean metric
and Log-Cholesky metric. The Euclidean case exhibits
significant swelling effect. Comparing to the Euclidean case, the Cholesky
distance \cite{Dryden2009} substantially alleviates the effect, but still suffers from noticeable swelling effect. In contrast,
the Log-Cholesky metric, as well as the affine-invariant metric and Log-Euclidean metric, is not subject to any swelling effect. \linremark{In addition, the affine-invariant, Log-Euclidean and Log-Cholesky geodesic interpolations showed in  \cref{fig:swelling-effect} are visibly indifferent. In fact, numerical simulations confirm that these three metrics often yield a similar Fr\'echet average of SPD matrices. For example, for a set of 20 randomly generated SPD matrices of dimension $m=3$, the expected relative difference in terms of squared Frobenius norm between the Log-Cholesky average and the affine-invariant(or Log-Euclidean) average is approximately $3.3\times 10^{-2}$.}

The computation
of geodesic interpolation for the Log-Cholesky metric is as efficient
as the one for the Log-Euclidean metric, since both metrics enjoy
a simple closed form for the Fr\'echet average of finite SPD matrices.
Moreover, it is even numerically stabler than  the Log-Euclidean metric
which is in turn stabler than the affine-invariant metric. On synthetic
examples of $3\times 3$ SPD matrices with the largest eigenvalue $10^{10}$ (resp. $10^{15}$) times
larger than the smallest eigenvalue, the Log-Cholesky metric is still
stable, while the Log-Euclidean one starts to deteriorate (resp. numerically
collapse).

\begin{figure}
\begin{centering}
\includegraphics[width=0.8\textwidth]{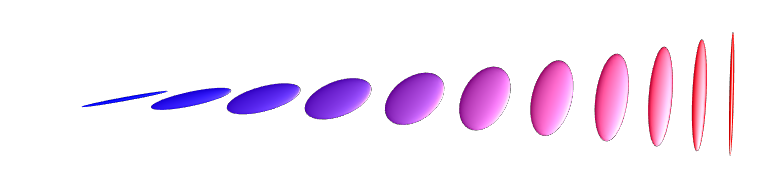}
\par\end{centering}
\begin{centering}
\includegraphics[width=0.8\textwidth]{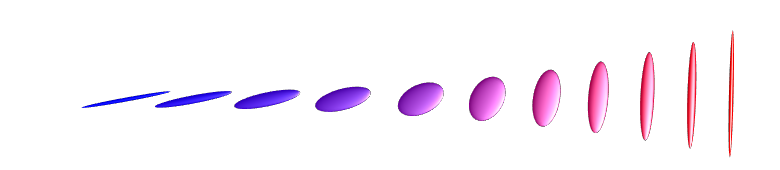}
\par\end{centering}
\begin{centering}
\includegraphics[width=0.8\textwidth]{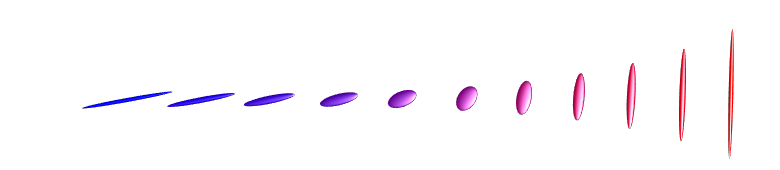}
\par\end{centering}
\begin{centering}
\includegraphics[width=0.8\textwidth]{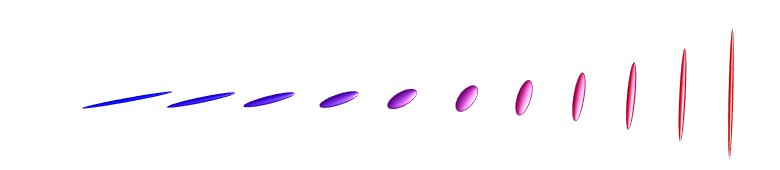}
\par\end{centering}
\begin{centering}
\includegraphics[width=0.8\textwidth]{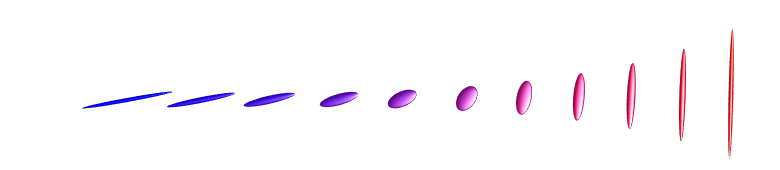}
\par\end{centering}
\caption{Interpolation of SPD matrices. Top: Euclidean linear interpolation.
The associated determinants are 5.40, 17.92, 27.68, 34.69, 38.93, 40.41,
39.14, 35.11, 28.32, 18.77, 6.46. Clearly, Euclidean interpolation
exhibits significant swelling effect. Second row: Cholesky interpolation.
5.40, 9.31, 13.12, 16.29, 18.43, 19.30, 18.80, 17.00, 14.08, 10.40,
6.46. The swelling effect in this case is reduced comparing to the
Euclidean interpolation. Third row: affine-invariant interpolation. Fourth row: Log-Euclidean interpolation. Bottom: Log-Cholesky geometric interpolation.
The associated determinants for the last three interpolations are the same: 5.40, 5.50, 5.60, 5.70, 5.80, 5.91,
6.01, 6.12, 6.23, 6.34, 6.46. There is no swelling effect observed for affine-invariant, Log-Euclidean and Log-Cholesky interpolation.}
\label{fig:swelling-effect}

\end{figure}

\section{Concluding remark\label{sec:Conclusion}}

We have constructed a new Lie group structure on SPD matrices via
Cholesky decomposition and a bi-invariant metric on it, termed Log-Cholesky
metric. Such structure and metric have the advantages of the Log-Euclidean
metric and affine-invariant metric. In addition, it has a simple
and closed form for Fr\'echet averages and parallel transport along
geodesics. For all of these metrics, Fr\'echet averages have the
same determinant and do not have swelling effect to which both
the Euclidean metric and the classic Cholesky distance are subject. Computationally,
it is much faster than its two counterparts, the  Log-Euclidean metric
and the affine-invariant metric. For computation of parallel transport,
it could be approximately 45 times faster than the Log-Euclidean metric
and 4 times faster than the affine-invariant one. The Log-Cholesky metric is also numerically stabler than these two metrics. 

In practice, which metric to
choose may depend on the context of applications while the presented
Log-Cholesky metric offers a choice alternative to existing metrics 
like the Log-Euclidean metric and the affine-invariant metric. For
big datasets, the advantage of the Log-Cholesky metric in computation is attractive. \linremark{However, for applications like \cite{Barachant2012} to which the congruence invariance property is central, the affine-invariant metric is recommended, since numerical experiments suggest that both the Log-Euclidean and Log-Cholesky metrics do not have the congruence invariance property that is enjoyed by the affine-invariant metric.}  

One shall also note that the Log-Cholesky metric
can be equivalently formulated in terms of upper triangular matrices.
\linremark{In the future, we plan to investigate other properties of Log-Cholesky means, e.g., their anisotropy and  relation to other geometric means or Fr\'echet means. We also plan to compare the performance of various metrics
in the study of brain functional connectivities
which are often characterized by SPD matrices.}

\section*{Acknowledgment}The author thanks Dr. Peng Ding and Dr. Qiang Sun for providing a potential application of the proposed Log-Cholesky framework when the latter approached the author for a geometric interpretation of their statistical model.

\bibliographystyle{siamplain}
\bibliography{manifold,misc}

\end{document}